\documentclass[11pt,addpoints]{article}
\usepackage{graphicx}
\graphicspath{{Figures/}}
\usepackage{amsfonts}
\usepackage{bbm}
\usepackage{amsmath}
\usepackage[table]{xcolor}
\usepackage[margin=2cm,bottom=2.5cm]{geometry}
\usepackage{changepage}
\usepackage{lscape}
\usepackage{multirow}
\usepackage{subcaption}
\usepackage{caption}
\usepackage{enumerate}
\usepackage{algorithm}
\usepackage{algorithmic}
\usepackage{xspace}
\usepackage{booktabs}
\usepackage{amsthm}

\newtheorem{theorem}{Theorem}
\newtheorem{corollary}{Corollary}

\newtheorem{remark}{Remark}

\usepackage[round]{natbib}
\usepackage{setspace}
\usepackage{hyperref}
\usepackage[affil-it]{authblk}

\newcommand{\red}[1]{{\color{red}#1}}

\newcommand{\ML}[1]{{\color{black}#1}}
\newcommand{\IL}[1]{{\color{black}#1}}
\newcommand{\MM}[1]{{\color{black}#1}}
\newcommand{\MS}[1]{{\color{black}#1}}
\newcommand{\KT}[1]{{\color{black}#1}}

\newcommand{\AlgName}[0]{{IBEG}}

\begin{document}
\title{An Exact Method for Fortification Games}

\author[1]{Markus Leitner \thanks{m.leitner@vu.nl}}
\affil[1]{Department of Operations Analytics, Vrije Universiteit Amsterdam, Netherlands}
	
\author[2]{Ivana Ljubi\'c \thanks{ljubic@essec.edu}}
\affil[2]{ ESSEC Business School of Paris, France}
		
\author[3]{Michele Monaci \thanks {michele.monaci@unibo.it}}
\affil[3]{Department of Electrical, Electronic and
	Information Engineering, University of Bologna,
	Italy}

\author[4]{Markus Sinnl \thanks {markus.sinnl@jku.at, } }
\author[4]{K\"ubra Tan{\i}nm{\i}\c{s} \thanks {kuebra.taninmis\_ersues@jku.at}}

\affil[4]{Institute of Production and Logistics Management, Johannes Kepler University Linz, Austria}

\date{}
\maketitle

\begin{abstract}
    
A fortification game (FG) is a three-level, two-player Stackelberg game, also known as defender-attacker-defender game, in which at the uppermost level, the defender selects some assets to be protected  from potential malicious attacks. At the middle level, the attacker solves an interdiction game by depreciating unprotected assets, i.e., reducing the values of such assets for the defender,
while at the innermost level the defender solves a recourse problem over the surviving or partially damaged assets. Fortification games have applications in various important areas, such as military operations, design of survivable networks, protection of facilities or power grid protection. 
%
In this work, we present an exact solution algorithm for FGs, in which the recourse problems correspond to (possibly NP-hard) combinatorial optimization problems. The algorithm is based on a new generic mixed-integer linear programming reformulation in the natural space of fortification variables. Our new model makes use of \emph{fortification cuts} that measure the contribution of a  given fortification strategy to the objective function value. These cuts are generated on-the-fly by solving separation problems, which correspond to (modified) middle-level interdiction games. %
We design a branch-and-cut-based solution algorithm based on fortification cuts, their \KT{strengthened} versions and other speed-up techniques.
We present a computational study using the knapsack fortification game and the shortest path fortification game. For the latter one, we include a comparison with a state-of-the-art solution method from the literature. Our algorithm outperforms this method and allows us to solve previously unsolved instances \MS{with up to 330\,386 nodes and 1\,202\,458 arcs} to optimality.
\end{abstract}


\newpage
\section{Introduction}

\IL{Fortification games (FGs), also known as \emph{defender-attacker-defender} (DAD) problems, have applications in various important areas, such as military operations, the design of survivable networks, protection of facilities, or power grid protection \citep{lozano2017backward,smith2020survey}.} A FG is a three-level, two-player Stackelberg game: At the third (innermost) level the defender wants to solve some optimization
problem (denoted as \emph{recourse} problem) which depends on some resources (\emph{assets}). 
At the second level, the attacker can select a subset of the assets to attack. 
Depending on the problem setting, this attack can either 
destroy the 
assets, or depreciate 
their usefulness for the defender. 
The goal of the attacker is to make the result of the optimization problem of the defender as worse as possible. 
Such a 
two-level attacker-defender problem described so far is known as interdiction game (IG) in which case
the actions of the attacker are called \emph{interdictions}. 
In FGs, the defender can prevent interdictions
at the first (outermost) level by 
fortifying some of its 
assets against potential attacks. Assets that are fortified cannot be interdicted by the attacker. 

\IL{We focus on a particular family of fortification games 
which is defined as follows.} Let $N$ be the set of assets, $f_i \geq 0$ be the cost for fortification of each asset $i \in N$, and $B_F$ be the fortification budget. 
Similarly, let $g_i\geq 0$ be the interdiction cost for each asset $i \in N$, and $B_I$ be the interdiction budget. 
Finally, let $d_i \geq 0$ denote the depreciation of an asset $i \in N$ due to interdiction. 
In the following, we will denote by $w$, $x$ and $y$ the incidence vector of  a fortification, an interdiction, and a recourse
strategy, respectively. 
Accordingly, we denote by $W=\{w\in \{0,1\}^{n}: \sum_{i\in N}f_i w_i\leq B_F\}$ the set of feasible fortifications and by
$X=\{x\in \{0,1\}^{n}:\sum_{i\in N}g_i x_i\leq B_I\}$ the set of possible interdictions if there would be no fortifications\KT{, where $n=|N|$}.
In addition, we let $X(w)=X\cap \{x\in \mathbb{R}^n: x\leq 1-w \}$ be the set of feasible interdictions for a given fortification strategy $w$,
and we denote by $Y \subseteq \{0,1\}^n$ the feasible region of the recourse problem, which is assumed to be non-empty. 
Using these definitions, the FGs studied in this work can be defined as 

\begin{align}
z^*= \, \min_{w\in W}\, \max_{x \in X(w)}\, &  \min_{y\in Y} \, \left\{ c^T y + \sum_{i\in N} d_i x_i y_i \right\} . \label{eq:P1} \tag{0-1 FG}
\end{align}

%


\begin{remark} Notice that according to our definition of \eqref{eq:P1}, interdiction decisions do not affect the feasible region $Y$ of the recourse problem, 
and each such decision is associated with a cost increase/penalty $d_i\geq 0$. 
However, without loss of generality, this definition also allows to model problems in which the 
attacker prohibits the use of the interdicted assets, 
i.e., where $Y(x)=Y \cap \{y \in  \mathbb{R}^n: y\leq 1- x\}$. In this case, the resulting problem 
\begin{align*}
z^*= \, \min_{w\in W}\, \max_{x \in X(w)}\, &  \min_{y\in Y(x)} \, c^T y
\end{align*}
can be reformulated as 
\begin{align*}
z^*= \, \min_{w\in W}\, \max_{x \in X(w)}\, &  \min_{y\in Y} \, \left\{  c^T y + \sum_{i\in N} \mathcal{M}_i x_i y_i \right\}
\end{align*}
for sufficiently large $\mathcal{M}_i$ (see, e.g., \cite{fischetti2019interdiction}), 
i.e., in problem \eqref{eq:P1} we set $d_i:=\mathcal{M}_i$.
\label{remark_P1}
\end{remark}
\MM{Finally, we notice that} 
\IL{
more general types of fortification games \MM{have been introduced in the literature.
For example, \citet{alderson2011solving} considered the case} in which the feasible region of the \ML{innermost} problem depends on $w$ as well 
(i.e., we have $Y(w,x)$ instead of $Y(x)$), \MM{whereas \citet{brown2006defending} addressed situations} 
in which the third level variables are not necessarily binary and there might exist some available capacity which is invulnerable to the attack 
(e.g., $y \le u_0 + u(1-x)$). These more general settings are out of scope of this paper, and we focus on a particular version defined by the problem \eqref{eq:P1}. 
}


\subsection{Contribution and Outline}

In this work, we present an exact solution algorithm for \eqref{eq:P1}. The algorithm is based on a new generic mixed-integer programming (MIP) formulation for \eqref{eq:P1} which makes use of valid inequalities, denoted as \emph{fortification cuts}. These cuts are used to measure the objective function value for given fortification strategies, and are generated on-the-fly by solving separation problems, which 
correspond to (modified) IGs. 

The detailed contribution and outline is summarized as follows.

\begin{itemize}
    \item We present a general solution framework for the  problem \eqref{eq:P1} whose recourse can be an arbitrary (possibly NP-hard) 
    combinatorial optimization problem with linear objective function and discrete variables. 
    Thus, we extend theory and methodology of the existing exact algorithms for fortification games, the majority of which  requires the recourse problems to be convex.
    \item We introduce a single-level MIP reformulation for \eqref{eq:P1} using fortification cuts, present alternative methods for \KT{strengthening} these
    cuts, and discuss their efficient separation.
    \item We propose further algorithmic enhancements, that allow to speed-up the separation problem, based on heuristic solutions of the interdiction problem.  
    \item We describe the application of our generic solution algorithm to two concrete problems, 
    namely the knapsack fortification game and the shortest path fortification game, and present an extensive computational
    study of the performance of our algorithm. This analysis is aimed at evaluating the contribution of different ingredients of our
    algorithm as well as 
    at comparing its effectiveness with other exact approaches from the literature.
 \end{itemize}

The paper is organized as follows.
In the remainder of this section, we provide an overview of previous and related work, whereas 
in Section \ref{section:FortCut} we introduce our reformulation and derive the fortification cut for a given intersection strategy.
Moreover, we discuss methods for \KT{strengthening} a given cut and analyze the associated separation problem.
Other general speed-up techniques that are used in our solution algorithm are presented in Section \ref{sec:algdetails}.
Sections \ref{section:KIF} and \ref{section:SPIF} describe two applications that can be modeled as \KT{an FG}.
A detailed computational study on both problems is given in Section \ref{section:Results}, where we 
discuss further, problem-specific implementation details, assess the efficiency of 
the various ingredients of the algorithm, and compare the results of our method with state-of-the-art exact approaches from the literature.
Finally, Section \ref{section:conclusions} draws some conclusions and reports possible new lines of research.

\subsection{Previous and Related Work}
Closely related to FGs are IGs,
which are \IL{two-player} two-level Stackelberg games used to model attacker-defender settings.  
The attacker, who acts first, has limited resources and an attack consists of disabling the defender's assets, reducing their capacity or increasing their cost. At the lower level, the defender solves the recourse problem over the set of surviving or partially damaged assets.   IGs arise in military applications \citep{brown2006defending},  in controlling the spread of infectious diseases
\citep{Assimakopoulos:1987,Shen-et-al:2012, Furini-et-al-OR:2021,Furini-et-al:2020, taninmis2020improved},
in counter-terrorism \citep{Wang-et-al:2016}, or in monitoring of communication networks \citep{FuriniLSZ21,FuriniLMS19}. Very often, IGs are defined over networks, in which the attacker reduces the capacities of nodes or edges, or even completely removes some of them from the network \citep{wood2010bilevel}. Some of the most famous examples of so-called network-interdiction games include interdiction of shortest paths \citep{israeli2002shortest} or  network flows \citep{Lim-Smith:2007,smith2008algorithms,Akgun:2011}. However, IGs turn out to be much more difficult if the recourse problem is NP-hard, like  maximum knapsack  \citep{caprara2016bilevel,fischetti2019interdiction,della2020exact} or  maximum clique \citep{FuriniLSZ21,FuriniLMS19}, making the associated IG  a $\Sigma_2^P$-hard problem \citep{Lodi-et-al:2014}.

In FGs, the defender tries to ``interdict'' the attacker, by anticipating the attacker's malicious activities, i.e., the defender tries to protect the most vulnerable assets before the attack is taking place. Even though there exists a large body of literature dedicated to IGs, see, e.g., the two recent surveys in \cite{smith2020survey,Kleinert-et-al:2021},
 there are very few articles dealing with more difficult three-level optimization problems arising in the context of FGs. 
The applications of FGs stem from similar settings as for the IGs. 
Three-level DAD models to protect electric power grids have been used by \citet{brown2006defending,yuan2014optimal,xiang2018improved,Lai-et-al:2019}, and \citet{Fakhry-et-al:2021}.   \citet{smith2007survivable} proposed to design networks that can survive network-flow attacks by using FGs. Similarly, \citet{Sarhadi-et-al:2017} used a DAD model for protection planning of freight intermodal transportation networks.
Other FG models for protecting transportation networks can be found in \citet{Jin-et-al:2015,StaritaScaparra:2016}, and \citet{Starita-et-al:2017}. 
In the context of supply chain networks and protection of facilities, we highlight DAD models presented by \citet{church2007protecting,scaparra2008exact}, and \citet{zheng2018exact}. 
Recently, a trilevel critical node problem used for limiting the spread of a viral attack in a given network has been
considered in \citet{Baggio-et-al:2021}.
We point out that our list of applications of FGs is far from being comprehensive, and we refer an interested reader to further references provided in the articles mentioned above.      

 When it comes to general techniques used to solve FGs,
  we distinguish between duality-based and reformulation-based methods. Duality-based methods  can be applied to FGs in which the recourse problem is a linear/convex program, so that after its  dualization (and potential linearization), the three-level model is turned into a bilevel 
  min-max mixed-integer  formulation \citep{brown2006defending}. The latter can then be solved using some of the state-of-the-art  solvers for mixed-integer bilevel optimization (see, e.g., \citet{fischetti2017new,fischettiMP,tahernejad2020branch}).
  Reformulation-based methods exploit a particular problem structure of the recourse or interdiction problem (by, e.g., enumerating all possible attack plans), so that an  equivalent bilevel or single-level model can be derived \citep{church2007protecting,scaparra2008bilevel,cappanera2011optimal}. 
  Heuristics
  are much more prevalent than exact methods in the existing literature on FGs. The most recent generic heuristics for FGs have been  proposed by \citep{Fakhry-et-al:2021}.
  
  Closest to our work is the backward sampling framework for FGs recently proposed by \citet{lozano2017backward}. The authors develop a cutting-plane method assuming 
  discrete fortification and interdiction strategies (i.e., all variables in the first two levels are assumed to be binary). 
  Thanks to the sampling of possible solutions of the third-level, the recourse problem can also be non-convex (non-linear, or involving discrete variables). For any given fortification strategy vector $w$, a subset of sampled third-level solutions is used as a basis  to derive valid lower and upper bounds of the embedded IG. The overall method is completed by an outer optimization over $w$ variables, where specific interdiction strategies are eliminated using no-good-cut-like covering inequalities.  
  
\section{Solution Framework}
\label{section:FortCut}

In this section, we first introduce a single-level reformulation of \eqref{eq:P1} in the space of fortification variables $w$ which uses the family of \emph{fortification cuts}. Afterwards, we present different methods for \KT{strengthening} these cuts and discuss methods for their separation.


\subsection{Single-level Reformulation and Fortification Cuts}
Let
$\Phi_I(x)$ be the value function of the interdiction game associated to a given interdiction strategy $x\in X$, i.e., 
\begin{align*}
\Phi_I(x)= \min_{y\in Y} \left\{ c^T y + \sum_{i\in N} d_i x_i y_i \right\}.
\end{align*}
Similarly, let $\Phi_F(w)$ be the value function of the fortification game for a given fortification strategy $w\in W$, i.e.,
\begin{align}
\Phi_F(w) &= \max_{x \in X(w)}\, \min_{y\in Y} \left\{ c^T y + \sum_{i\in N} d_i x_i y_i \right\} = \max_{x \in X(w)} \Phi_I(x) \notag \\
&= \max_{x \in X} \left\{ \Phi_I(x) - \sum_{i\in N} M_i x_i w_i \right\}. \label{eq:eq7}
\end{align}
The last equations holds for sufficiently large coefficients $M_i$, due to the fact that the linking constraints between the fortification and subsequent interdiction are given by $x \le 1 - w$. 
Using this notation, problem~\eqref{eq:P1} can equivalently be written as
\begin{align}
z^*= \, \min_{w\in W}\, &\theta \label{eq:reform_obj}\\
&\theta \geq \Phi_F(w). \label{eq:reform_feas}
\end{align}

Equation \eqref{eq:eq7} implies that, for any $w \in W$ and $\hat{x}\in X$, we have 
\begin{equation}
    \Phi_F(w)\geq \Phi_I(\hat{x}) - \sum_{i\in N} M_i \hat{x}_i w_i,
\label{eq:PhiF}
\end{equation}
which leads to the single level reformulation of the fortification game 
\begin{align}
z^*= \, \min_{w\in W}\, &\theta \label{eq:fortification-obj} \\
&\theta \geq \Phi_I(\hat{x}) - \sum_{i\in N} M_i \hat{x}_i w_i \hspace*{1cm} \forall \hat{x} \in X. 
\label{eq:fortification}
\end{align}

Given the possible exponential number of inequalities \eqref{eq:fortification}, this reformulation is suitable for being solved using a 
cutting-plane or a branch-and-cut approach in which they are 
initially omitted and then 
added on-the-fly when
violated. The resulting formulation in which some of the constraints \eqref{eq:fortification} are relaxed, is referred to as \emph{relaxed master problem}. 
There are two potential drawbacks in using formulation \eqref{eq:fortification-obj}-\eqref{eq:fortification} in practical applications. First, if not carefully chosen, the values of $M_i$ can lead to very weak dual bounds and potentially underperforming branch-and-bound trees. Second, to separate inequalities \eqref{eq:fortification} one has to solve the middle-level IG, a problem which can be $\Sigma_2^P$-hard, for NP-hard recourse problems. To address the first issue, we show in Theorem \ref{theo:fort_cut}  
that tight values for the $M_i$ coefficients can be easily derived, no matter the structure of the recourse problem. We then show how these 
\emph{fortification cuts}, can be \KT{strengthened} to tighten the dual bounds. We address the separation issue in Section \ref{section:separation}.

\begin{theorem}
Constraints \eqref{eq:fortification} can be replaced with the following \emph{fortification cuts}:
\begin{equation}
\theta \geq \Phi_I(\hat{x}) - \sum_{i\in N} d_i \hat{x}_i w_i \hspace*{1cm} \forall \hat{x} \in X. \label{eq:fort_cut}
\end{equation}
\label{theo:fort_cut}
\end{theorem}

\begin{proof}

\KT{We need to} show that $\Phi_F(w)\geq \Phi_I(\hat x) - \sum_{i\in N} d_i \hat{x}_i w_i$ for all $w\in W$ and $\hat{x} \in X$, i.e., that \eqref{eq:PhiF} holds for $M_i=d_i$. 

This is straightforward in case 
$\hat{x} \in X(w)$ since, by definition, we have $\hat{x}_i w_i = 0$ for all $i \in N$, and
\eqref{eq:PhiF} reduces to $\Phi_F(w)\geq \Phi_I(\hat{x})$,
implying inequality \eqref{eq:fort_cut}. 
    
If instead $\hat{x} \not \in X(w)$, we 
consider 
interdiction strategy $x^\prime$ defined by selecting those assets in $\hat{x}$ that are
    not fortified by $w$, i.e.,
    $x^\prime_i = \hat{x}_i (1-w_i)$, for all $i\in N$.
	Since $\hat{x} \in X$, we have $x^\prime \in X(w)$ and \eqref{eq:PhiF}, therefore, $\Phi_F(w)\geq \Phi_I(x^\prime)$ holds.
    We complete the proof by showing that 
    \begin{equation}
    \Phi_I(x^\prime)\geq \Phi_I(\hat{x}) - \sum_{i\in N} d_i \hat{x}_i w_i.
    \end{equation}
    By contradiction, assume that 
    $\Phi_I(\hat{x}) - \sum_{i\in N} d_i \hat{x}_i w_i > \Phi_I(x^\prime)$, and 
    let $y^\prime\in Y$ be an optimal recourse strategy 
    for interdiction strategy $x^\prime$. 
    Using the definition of value function $\Phi_I(x^\prime)$, we obtain
\begin{equation}
  \Phi_I(\hat{x}) - \sum_{i\in N} d_i \hat{x}_i w_i >
  c^T y^\prime + \sum_{i\in N} d_i x^\prime_i y^\prime_i. \label{eq:z_hat vs z_prime}
\end{equation}
Using notations 
$N_1 = \{i\in N: \hat{x}_i \, w_i=1\} = 
\{i \in N: \hat{x}_i=1, x^\prime_i=0 \}$,
and $N_2=N\setminus N_1$, 
inequality~\eqref{eq:z_hat vs z_prime} can be 
rewritten as 
\begin{equation}
\Phi_I(\hat{x}) > 
  c^T y^\prime + 
  \sum_{i\in N_2} d_i x^\prime_i y^\prime_i + 
  \sum_{i\in N_1} d_i,
\end{equation}
since $x^\prime_i = 0$ for all $i \in N_1$,
and $\hat{x}_i w_i = 1$ for all $i \in N_1$.
Since 
$x^\prime_i = \hat{x}_i$ for all $i\in N_2$, and 
$d_i\geq d_i \hat{x}_i y^\prime_i$ for all $i \in N$, it follows that
\begin{equation}
\Phi_I(\hat{x}) > 
  c^T y^\prime + 
  \sum_{i\in N_2} d_i \hat{x}_i y^\prime_i + 
  \sum_{i\in N_1} d_i \hat{x}_i y^\prime_i = 
  c^T y^\prime + 
  \sum_{i\in N} d_i \hat{x}_i y^\prime_i
\end{equation}
which shows that 
$y^\prime$ is a feasible
recourse strategy
for interdiction strategy 
$\hat{x}$ whose cost is strictly smaller than 
$\Phi_I(\hat{x})$. This contradiction concludes the
proof.
\end{proof}


\subsection{\KT{Strengthening} the Fortification Cuts with Enumeration}
\label{section:Lift_enum}

\MM{
In this section, we introduce sufficient conditions under which a vector of coefficients 
produces a valid cut strengthening a given fortification cut.
}

\begin{theorem}\label{th:enum}
Let $\hat{x}\in X$ be a feasible interdiction
strategy and 
denote by $S(\hat{x})=\{i\in N: \hat{x}_i=1\}$ the 
associated set of interdicted assets.
Let $\tilde{d}\in \mathbb{R}^n$ be a vector 
satisfying the following conditions:
\begin{enumerate}
\item $\tilde{d}_i \leq d_i$,  for all $i\in N$,
\item $\displaystyle{\sum_{i\in P} \tilde{d}_i} \geq  \Phi_I(\hat{x}) - \Phi_I(x^\prime)$ 
for each $P\subseteq S(\hat{x})$ such that there exists a {feasible} fortification strategy 
$w \in W: w_i=1$ for all $i \in P$,
and
$x^\prime$ is the interdiction strategy with 
$x^\prime_i = 0$ for all $i \in P$, and
    $x^\prime_i=\hat{x}_i$ otherwise.
\end{enumerate}

Then, the 
inequality
\begin{equation}
\theta \geq \Phi_I(\hat{x}) - \sum_{i\in N} \tilde{d}_i \hat{x}_i w_i
\label{eq:liftEnum}
\end{equation}
is 
valid 
for
\eqref{eq:P1}
and dominates the fortification cut~\eqref{eq:fort_cut} for interdiction strategy $\hat{x}$.
\label{theo:liftEnum}
\end{theorem}

\begin{proof}
Let $\hat{x}\in X$ be a feasible interdiction strategy, 
$\tilde{d}\in \mathbb{R}^n$ be a vector satisfying 
Conditions~1 and 2 above, and $w\in W$ be any feasible 
fortification strategy. Define $P=\{i\in N:\hat{x}_i=1, w_i=1\} \subseteq S(\hat{x})$, which clearly satisfies the description in Condition 2.
In addition, consider an interdiction strategy $x^\prime$ such that $x^\prime_i=\hat{x}_i(1-w_i)$ for all $i \in N$. 
By definition of $P$ and using Condition 2, we get
\begin{equation}
\Phi_I(\hat{x}) - \sum_{i\in N} \tilde{d}_i \hat{x}_i w_i =
\Phi_I(\hat{x}) - \sum_{i\in P} \tilde{d}_i \leq \Phi_I(x^\prime) \leq \theta, \label{eq:lift_enum_proof}
\end{equation}
where the last inequality follows from the fact that $\theta\geq \Phi_F(w)\geq \Phi_I(x^\prime) $ since $x^\prime \in X(w)$, as in the proof of Theorem \ref{theo:fort_cut}.
Thus, 
\KT{\eqref{eq:liftEnum} is valid} since \eqref{eq:lift_enum_proof} holds for all $w\in W$.

Finally, Condition 1 ensures that cut \eqref{eq:liftEnum} dominates
its counterpart \eqref{eq:fort_cut} associated with 
the same interdiction strategy $\hat{x}$.
\end{proof}
 
\KT{
The theorem above provides conditions for determining stronger cuts. While the trivial choice of $\tilde{d}=d$ is always feasible with respect to these conditions, a tighter cut could potentially be obtained by finding a vector $\tilde{d}$ which is minimal with respect to some norm. We note that for a given interdiction strategy $\hat{x}$, checking 
\MM{whether a given vector $\tilde{d}$ satisfies to Condition 2}
may be computationally challenging. Indeed, one has to evaluate 
the recourse value of the attacker solution 
$S(\hat{x}) \setminus P$ for each $P\subseteq S(\hat{x})$ whose elements can be contained in a feasible fortification strategy. In Section~\ref{section:lifting details}, we describe a heuristic approach to efficiently obtain a non-trivial vector $\tilde{d}$ satisfying the conditions of Theorem \ref{th:enum}. 
}


\begin{remark}
The right-hand-side of Condition 2 in Theorem \ref{theo:liftEnum} can be
replaced by 
$\min\{\Phi_I(\hat{x}) - \underline{\Phi}_I(x^\prime), \sum_{i\in P} d_i\}$, where $\underline{\Phi}_I(x^\prime)$ is a lower bound on $\Phi_I(x^\prime)$. 
For the problems whose recourse level is computationally
hard to solve, a dual bound could be used to produce a 
faster \KT{strengthening} procedure. 
Although the resulting cut may be weaker than the one obtained using $\Phi_I(x^\prime)$, it is still as good as \eqref{eq:fort_cut} since $\tilde{d}_i \leq d_i$, for all $i\in N$. 
\label{remark_liftEnum}
\end{remark}

\subsection{\KT{Strengthening} the Fortification Cuts with a Lower Bound}
\label{section:Lift_LB}

In this section we show an alternative and computationally less expensive 
way to obtain \KT{stronger} fortification cuts. The theorem below 
assumes that a valid lower bound on $z^*$ is available. For $\alpha \in \mathbb{R}$, we define $(\alpha)^+= \max\{\alpha, 0\}$.

\begin{theorem}\label{th:lower}
Let $\hat{x}\in X$ be a feasible interdiction
strategy and $\underline{z}$ be a lower bound on $z^*$. Then, the inequality
\begin{equation}
    \theta \geq \Phi_I(\hat{x}) - \sum_{i\in N}  \min\{(\Phi_I(\hat{x}) - \underline{z})^+, d_i\} \hat{x}_i w_i
    \label{eq:liftBound}
\end{equation}
is valid for \eqref{eq:P1} and dominates the fortification cut~\eqref{eq:fort_cut} for interdiction strategy $\hat{x}$.
\end{theorem}

\begin{proof}
Since $\underline{z}$ is a lower bound on $z^*=\min_{ w\in W}  \Phi_F(w)$, we have $\Phi_F(w)\geq \underline{z}$, for all $w\in W$. 
Let $L=\{i\in N : \hat{x}_i=1, \Phi_I(\hat{x}) - \underline{z} <d_i\}$ be the set of indices for which the 
coefficient in fortification cut~\eqref{eq:liftBound} is smaller than in \eqref{eq:fort_cut}.
It is clear that \eqref{eq:liftBound} is equivalent to $\eqref{eq:fort_cut}$ if $w_i=0$ for all $i\in L$.
Otherwise, \begin{equation}
    \Phi_I(\hat{x}) - \sum_{i\in N}  \min\{(\Phi_I(\hat{x}) - \underline{z})^+, d_i\} \hat{x}_i w_i \leq \Phi_I(\hat{x}) - (\Phi_I(\hat{x}) - \underline{z})^+ \leq \underline{z},
    \label{eq:liftLB-proof}
\end{equation} 
which leads to $\Phi_F(w) \geq \underline{z} \geq \Phi_I(\hat{x}) - \sum_{i\in N}  \min\{(\Phi_I(\hat{x}) - \underline{z})^+, d_i\} \hat{x}_i w_i$. 
\end{proof}

\begin{remark}
Notice that $z^* \geq \min_{y\in Y} \, c^T y$, i.e., the objective value of the fortification problem cannot be smaller \KT{than} the recourse objective value \KT{in case of no interdiction}. Thus, a valid lower bound on $z^*$ to be used for deriving \KT{strengthened} cuts \eqref{eq:liftBound} 
can be computed as $\underline{z}=\min_{y\in Y} \, c^T y$.  
\label{remark_liftLB}
\end{remark}

The following corollary shows that the results provided in Theorems \ref{th:enum} and \ref{th:lower}
above can be combined to obtain 
a valid cut. It follows from the validity of \eqref{eq:liftEnum} and inequality \eqref{eq:liftLB-proof}.

\begin{corollary}
	Let $\hat{x}\in X$ be a feasible interdiction strategy,
	$\tilde{d}\in \mathbb{R}^n$ be a vector satisfying the 
	conditions in Theorem \ref{theo:liftEnum}, 
	and $\underline{z}$ be a lower bound on $z^*$. 
	Then, the inequality 
	\begin{equation}
	\theta \geq \Phi_I(\hat{x}) - \sum_{i\in N}  \min\{(\Phi_I(\hat{x}) - \underline{z})^+, \tilde{d}_i\} \hat{x}_i w_i
	\end{equation}
	is valid for 
	\eqref{eq:P1}. 
	\label{theo:combined_lifting}
\end{corollary}

%

\begin{remark}

When $\Phi_I(\hat{x}) - \underline{z}$ is small compared to \KT{the} original cut coefficient $d$, it is likely that most of the final coefficients after combined \KT{strengthening} will be equal to $\Phi_I(\hat{x}) - \underline{z}$, which renders the effort to initially compute $\tilde{d}_i$ values unnecessary. In this case, one could opt for the lower bound based \KT{strengthening} which can be defined in constant time if $\underline{z}$ is available. 
\end{remark}

\subsection{Separation of Fortification Cuts}
\label{section:separation}

As already mentioned, our approach for solving the problem \eqref{eq:P1} is a branch-and-cut algorithm 
based on reformulation 
\eqref{eq:fortification-obj}
with (\KT{strengthened}) fortification cuts \eqref{eq:fort_cut}.
At each node of the branch-and-cut tree, 
valid inequalities are added on-the-fly, when 
violated, to ensure correctness of the algorithm, improve the dual bound at the node, and
possibly allow fathoming. 
Although in principle one could add to the formulation 
{\em any} valid violated inequality, a cut that is
maximally violated is typically sought. 
Thus, given a solution, say $(w^*,\theta^*)$, for the linear programming (LP)-relaxation
of the current model, one is required to solve the separation problem, \KT{which is an IG,}
\begin{equation}\label{eq:thetaF}
\Phi_F(w^*) = \max_{x \in X(w^*)}\, \min_{y\in Y} \left\{
c^T y + \sum_{i\in N} d_i x_i y_i \right\}
\end{equation}
and check whether $\Phi_F(w^*) > \theta^*$ or not.
Notice that this 
separation problem itself is a mixed-integer bilevel linear program, which is at least NP-hard and \KT{possibly} even $\Sigma_2^P$-hard \citep{caprara2014study} in case of an integer recourse problem.
%
Separation problem 
\eqref{eq:thetaF} can be reformulated 
as follows \KT{using Benders decomposition 
as done by \citet{israeli1999system}}.
\begin{align}
\text{(SEP)} \qquad
\Phi_F(w^*) = \max_{x,t} \quad & t \\
& t \leq  c^T \hat{y} + \sum_{i\in N}d_{i} \hat{y}_{i} x_{i}
\hspace*{5ex} \forall \hat{y} \in \hat{Y} \label{eq:interCut} \\
& \sum_{i\in N}g_i x_i\leq B_I  \label{eq:SepBudget}\\
& x_{i} \leq 1-w^*_{i} \hspace*{13ex} \forall i\in N  \label{eq:SepFortification}\\
&  x\in \{0,1\}^{n}\label{eq:SepBinary}
\end{align} 
Here $\hat Y$ denotes the set of extreme points of the convex hull of the feasible solutions of the recourse problem, see, e.g., \citet{wood2010bilevel,fischetti2019interdiction}.
In our generic framework, we solve (SEP) using a branch-and-cut scheme in which we add violated 
\emph{interdiction cuts}~\eqref{eq:interCut} 
on the fly, for both integer and fractional values of $x$ variables. 
Solving the separation problem 
yields a violated fortification cut 
if $\Phi_F(w^*) > \theta^*$, while
one can conclude that $(w^*,\theta^*)$ is a feasible 
solution to \eqref{eq:reform_obj}--\eqref{eq:reform_feas} otherwise.

\paragraph{\KT{Strengthening} the interdiction cut} 
Given the computational complexity of the separation problem (SEP), one may be interested in determining
a violated, but not necessarily maximally violated fortification cut. To this aim, it is enough to find a feasible solution to (SEP) 
having an objective value strictly larger than $\theta^*$. 
Hence, we propose to reformulate the separation problem by imposing 
an \emph{artificial} upper bound, a real number $\ell > \theta^*$, on its optimal objective value.  To this end, the interdiction cuts \eqref{eq:interCut} can be \KT{strengthened} as shown in the following theorem.

 \begin{theorem}
 Given a solution $(w^*,\theta^*)$ and a real number $\ell > \theta^*$, consider the following formulation: 
 \begin{align}
\text{(SEP-L)} \qquad 
 \Phi^\ell_F(w^*) = \max_{x,t} \quad & t \notag \\
&  t \leq c^T\hat{y} + \sum_{i\in N}\min\{(\ell - c^T\hat{y})^+, d_i\} \hat{y}_{i} x_{i}  
\hspace*{5ex} \forall \hat{y} \in \hat{Y} \label{eq:lifted_interCut} \\
& \eqref{eq:SepBudget}-\eqref{eq:SepBinary} \notag
\end{align}    
Let $(x^*,t^*)$ be an optimal solution to (SEP-L). If $\Phi^\ell_F(w^*) > \theta^*$, then $\Phi_F(w^*) > \theta^*$ 
and $x^*$ gives a fortification cut \eqref{eq:fort_cut} which is violated at $(w^*,\theta^*)$. Otherwise, $(w^*,\theta^*)$ is a feasible solution to \eqref{eq:P1}.
 \end{theorem}
\begin{proof}
\KT{

Observe that the problem (SEP-L) is always feasible and bounded, hence there exists an optimal solution $(x^*,t^*)$ with the associated objective value $\Phi^\ell_F(w^*)$, and $t^*=\Phi^\ell_F(w^*) \le \Phi_F(w^*)$. Therefore, if $\Phi^\ell_F(w^*) > \theta^*$, then $\Phi_F(w^*) > \theta^*$, and $x^*$ gives a violated cut since $\Phi(x^*)\geq \Phi^\ell_F(w^*)$. 
Otherwise, i.e., if $\Phi^\ell_F(w^*) \le \theta^*$, then there exists some $\hat{y}$ for each $x\in X(w^*)$ that makes the RHS of \eqref{eq:lifted_interCut} less than or equal to $\theta^*$. Moreover, \eqref{eq:lifted_interCut} and \eqref{eq:interCut} are identical at such $\hat{y}$ and $x$ pairs, which means $\Phi_F(w^*) =\Phi^\ell_F(w^*)\le \theta^*$. 
In this case there exists no violated cut at $(w^*,\theta^*)$.
}

\end{proof}

In other words, this \KT{strengthened} interdiction cut may project the true objective value of the recourse problem to a smaller value, but if the optimal objective value $\Phi_F(w^*)$ of (SEP) is strictly greater than $\theta^*$, then the optimal objective value $\Phi^\ell_F(w^*)$ of (SEP-L) is also greater than $\theta^*$. As a result, if (SEP-L) yields a 
solution $(\hat{x},\hat{t})$ with $\hat{t}> \theta^*$, then $\hat{t}$ could be less than $\Phi_I(\hat{x})$ which is the constant part of the fortification cut \eqref{eq:fort_cut}. Thus, the recourse problem should be solved for $\hat{x}$ one more time to obtain $\Phi_I(\hat{x})$. Note that  \cite{lozano2017backward} also propose a \KT{strengthening} procedure for interdiction cuts, however using a global upper bound on $z^*$.
Although this approach is valid in our case too, our \KT{strengthening} is more aggressive since $\ell$ is not necessarily an upper bound, but is treated as one.

In Section \ref{section:separationProcedure}, we propose various additional strategies to speed up the separation procedure.

\section{Algorithmic Details for an Efficient Implementation}\label{sec:algdetails}

In this section we address the most important algorithmic elements necessary for an  efficient implementation of our solution framework.

\subsection{Initialization}

We set an initial lower bound on $\theta$ by solving the recourse problem with no interdiction as mentioned in Remark \ref{remark_liftLB}, i.e., we add the constraint $\theta\geq \min_{y\in Y} \, c^T y$ to our model.

\KT{

We also initialize our model with a set of initial fortification cuts.
These cuts are produced using the interdiction strategies computed by Algorithm \ref{alg:InitialCuts}. This algorithm takes an optimal recourse solution $y^0$ in the absence of interdiction as input, and iterates over the assets used in this solution. For each such asset $i$, an interdiction strategy is obtained via first interdicting $i$, then deciding the remaining interdiction actions in a greedy fashion. In the greedy part, we pick one element with maximum depreciation-cost ratio among all those that can still be added to the current interdiction strategy without exceeding the interdiction budget. Notice that the number of initial cuts, i.e., the interdiction strategies that Algorithm \ref{alg:InitialCuts} generates, depends on the support size of the initial recourse solution $y^0$.

\begin{algorithm}[h]
\caption{\texttt{InitialFortificationCuts}}
\textbf{Input:}  An optimal recourse solution $y^0$ in case of no interdiction \\
\textbf{Output:} A set of initial fortification cuts \eqref{eq:fort_cut}
\begin{algorithmic} [1]
\FOR {each $i\in N: y^0_i=1$}
\STATE Initialize the 
interdiction strategy $\hat{x}\leftarrow e_i$, define $\hat{B}_I  \leftarrow B_I-g_i$
\STATE Initialize the candidate set for interdiction $R_N\leftarrow N\setminus\{i\}$
\WHILE{$R_N\neq \emptyset$}
\STATE Solve the recourse problem for $\hat{x}$ to obtain an optimal solution $y^*$ \label{alg:gi:recourse}
\STATE Obtain the new candidate set $R_N=\{i\in N:  \hat{x}_i=0 ,\, y^*_i=1, g_i\leq \hat{B}_I\}$ 
\IF{$R_N\neq \emptyset$}
\STATE Select $i^*=\arg \max_{i\in R_N} \frac{d_i}{g_i}$
\STATE Set $\hat{x}_{i^*}\leftarrow 1$, and update the budget $\hat{B}_I \leftarrow \hat{B}_I- g_{i^*}$
\ENDIF
\ENDWHILE
\STATE {Generate a fortification cut \eqref{eq:fort_cut} for $\hat{x}$}
\ENDFOR
\end{algorithmic}
\label{alg:InitialCuts}
\end{algorithm}

}

\subsection{On Implementing the Separation of Fortification Cuts}
\label{section:separationProcedure}

We solve the formulation \eqref{eq:fortification-obj}-\eqref{eq:fortification} by means of a branch-and-cut algorithm, in which \KT{feasible} solutions $(w^*,\theta^*)$ of the relaxed master problem \KT{such that $w^*$ is integer} are separated on the fly.
Let $(w^*,\theta^*)$ be a feasible solution of the relaxed master problem at the current branching node, e.g., the  solution of the LP-relaxation at the current node or a heuristic solution, \KT{where $w^*$ is integer}.
Generating a violated fortification cut associated to $(w^*,\theta^*)$ (if such exists) corresponds to finding $\hat{x}\in X(w^*)$ such that $\Phi_I(\hat{x})>\theta^*$. 

\paragraph{Greedy Integer Separation}
To search for such $\hat x$, we first invoke the greedy heuristic described in Algorithm \ref{alg:GreedyInterdiction}.
This procedure iteratively defines an interdiction strategy (initially $\hat{x} = \Vec{0}$) and 
a candidate set for interdiction (initially including all assets). At each iteration, the recourse problem 
associated with the currently interdicted assets is solved, the candidate set is updated accordingly, 
and an asset is selected and added to the interdiction strategy. The procedure ends when the candidate set for
interdiction is empty, meaning that no asset that has been selected in the recourse strategy can be added to the interdiction strategy.
\begin{algorithm}[h]
\caption{\texttt{GreedyInterdiction}}
\textbf{Input:}  Current solution $(w^*,\theta^*)$ of the relaxed master problem \\
\textbf{Output:} $\hat{x}\in X(w^*)$, possibly giving a violated 
cut~\eqref{eq:fort_cut} 
\begin{algorithmic} [1]
\STATE Initialize the 
interdiction strategy $\hat{x}\leftarrow \Vec{0}$ and 
remaining interdiction budget 
$\hat{B}_I\leftarrow B_I$
    \STATE Initialize the candidate set for interdiction $R_N\leftarrow N$
\WHILE{$R_N\neq \emptyset$}
\STATE Solve the recourse problem for $\hat{x}$ to obtain an optimal solution $y^*$ \label{alg:gi:recourse}
\STATE Obtain the new candidate set $R_N=\{i\in N:  \hat{x}_i=0 ,\, y^*_i=1, g_i\leq \hat{B}_I\}$ 
\IF{$R_N\neq \emptyset$}
\STATE Select $i^*\in \arg \max_{i\in R_N} \frac{d_i(1-w^*_i)}{g_i}$
\STATE Set $\hat{x}_{i^*}\leftarrow 1$, and update the budget $\hat{B}_I \leftarrow \hat{B}_I- g_{i^*}$
\ENDIF
\ENDWHILE
\STATE Return $\hat{x}$
\end{algorithmic}
\label{alg:GreedyInterdiction}
\end{algorithm}

If this \KT{\texttt{GreedyInterdiction}} heuristic does not produce 
a violated cut, the separation problem (SEP) described in Section \ref{section:separation}
is solved by means of a branch-and-cut algorithm. We now present some strategies that are used
to improve the performance when solving the model (SEP).

\paragraph{Solution Limit as a Stopping Criterion}
As already observed, any feasible solution to (SEP) with an objective value larger 
than $\theta^*$ yields a violated cut. Accordingly, one can avoid the computation of a maximally violated cut, and halt the execution of the 
enumerative algorithm for (SEP) as soon as the incumbent objective value becomes larger than $\theta^*$. In general, a solution limit $s_{max}\geq 1$ can be imposed to prematurely terminate the solution of (SEP) once at least $s_{max}$ incumbent solutions with an objective value larger than $\theta^*$ are found. 

 \paragraph{Setting a Lower Cutoff Value} Another strategy that could be helpful in solving (SEP) is to set the lower cutoff value of the used MIP solver to \KT{$\theta^*+\epsilon$ where $\epsilon$ is a small constant}. In this way, one can prune all the nodes with an upper bound less than \KT{$\theta^*+\epsilon$}, i.e., nodes that cannot produce a violated interdiction cut.
 \KT{If there is no $x\in X(w^*)$ with $\Phi_I(x) \geq \theta^*+\epsilon$, then (SEP) becomes infeasible due to the lower cutoff value, which can be considered as an indicator of bilevel feasiblity of $(w^*,\theta^*)$ since $X(w)\neq \emptyset$ for any value of $w$. In that case, no cuts are added. 
 With this lower cutoff strategy, whenever the incumbent fortification strategy is updated, the associated interdiction strategy is not readily available. Therefore, when an optimal/incumbent solution $(w^*, \theta^*)$ to \eqref{eq:P1} is found, it is possible to obtain the optimal attacker response by iterating over all the fortification cuts added to the master problem, i.e., previously obtained interdiction strategies $\hat{x}$. Any $\hat{x}\in X(w^*)$ with $\Phi_I(\hat{x})=\theta^*$ is an optimal response to $w^*$. If none is found, the attacker problem is solved optimally for $w^*$.}
 %



While solving (SEP), both integer and fractional $x^*$ are separated. The lower cutoff and the solution limit strategies described above are employed to speed up the separation process. Based on our preliminary experiments, we set the solution limit parameter $s_{max}$ to one.
%
Finally, all feasible solutions to the recourse problem obtained after solving (SEP) in previous iterations are stored and added as initial cuts 
of the form \eqref{eq:interCut} in future solving attempts. 

We remark that  the same enhancements can be applied to the alternative separation  model (SEP-L) presented in Section \ref{section:separation}.

\subsection{On Implementing the \KT{Strengthening} of the Fortification Cut}
\label{section:lifting details}

\begin{algorithm}[b!]
\caption{\texttt{Enumerative\KT{Strengthening}}}
\textbf{Input}: $\hat{x}\in X$ \\
\textbf{Output}: Valid  coefficients $\tilde{d}$ for the \KT{strengthened} fortification cut \eqref{eq:liftEnum}
\begin{algorithmic}[1]
\STATE Initialize the \KT{strengthened} cut coefficients as $\tilde{d}\leftarrow \Vec{0}$ 
\FORALL{$\emptyset\ne P\subseteq S(\hat{x}) \mbox{ s.t. } \sum_{i\in P} f_i\le B_F$}
\STATE Define attacker solution $x^\prime$ such that $x^\prime_i = 0$ for all $i \in P$, and $x^\prime_i=\hat{x}_i$ otherwise
\STATE Solve the recourse problem to obtain $\Phi_I(x^\prime)$, compute $\Delta_P=\Phi_I(\hat{x}) - \Phi_I(x^\prime)$ \label{alg:el:recourse}
\WHILE{ $\sum_{i\in P}\tilde{d}_i < \Delta_P$}
\STATE Randomly pick $i\in P$ such that $\tilde{d}_i<d_i$, set $\tilde{d}_i\leftarrow \tilde{d}_i + \min \{ d_i-\tilde{d}_i, \Delta_P- \sum_{i\in P}\tilde{d}_i \}$
\ENDWHILE
\IF{$\tilde{d}=d$}
\STATE \textbf{return} $\tilde{d}$
\ENDIF
\ENDFOR
\STATE \textbf{return}  $\tilde{d}$
\end{algorithmic}
\label{alg:EnumLift}
\end{algorithm}


In order to strengthen a fortification cut based on Theorem \ref{theo:liftEnum}, one may solve an LP, e.g., $\min_{ \tilde{d} \in \mathbb{R}^n}$ $\sum_{i\in S(\hat{x})} \tilde{d}_i$ subject to the conditions of Theorem \ref{theo:liftEnum}, after computing all the RHS values appearing in Condition 2. Our preliminary experiments showed that solving this LP was rather time consuming. We thus developed a heuristic method, reported in Algorithm \ref{alg:EnumLift}, which is more time efficient and which produces cuts that are usually very similar to the cuts obtained by solving the LP.
This procedure receives an interdiction strategy $\hat{x}$ as input and returns \KT{strengthened} coefficients $\tilde{d}$ \KT{that satisfy the condition in Theorem \ref{theo:liftEnum}. 

The heuristic works as follows: initially}, all coefficients are set to zero. Then,
all sets $P \subseteq S(\hat{x})$ are considered and, 
for each set, the corresponding attacker solution $x^\prime$ is 
defined. Then, the cost reduction $\Delta_P=\Phi_I(\hat{x})-\Phi_i(x^\prime)$
with respect to the recourse cost for $\hat{x}$ is split among $\tilde{d}$ coefficients, provided these figures satisfy Condition 1 of Theorem \ref{theo:liftEnum}.
Note that unnecessary computations can be avoided by considering only the subsets $P$ such that there is a feasible recourse solution with $y_i=1$, for all $i\in P$, as the combined effect of fortifying any two assets on the recourse objective can be larger than the sum of their individual effects only if they can simultaneously be used in a recourse solution.
In order to do that, one may exploit the structures of the recourse problems, if known. This is explained in more detail 
in Section \ref{section:Results} within the problem specific implementation aspects of the two applications we describe next. 
In addition, any iteration could be halted whenever the 
sum of the current $\tilde{d}$ coefficients is larger than the maximum value $\Delta_P$ could take, i.e., $\Phi_I(\hat{x})- \min_{y\in Y} \, c^T y$.

For what concerns the \KT{strengthening} of fortification cuts based on a lower bound \KT{as described in} Section \ref{section:Lift_LB}, we initialize the \KT{global} lower bound 
as $\underline{z}=\min_{y\in Y} \, c^T y$. \KT{This value is updated} dynamically as the best known lower bound \KT{obtained from the LP relaxation objective value} changes in the branching tree.
Whenever the fortification cut is \KT{strengthened} with $\underline{z}$, the cut is \emph{globally valid} and added as a global cut. 
In addition, 
let $(w^*,\theta^*)$ be an optimal solution of the  LP-relaxation at the current branching node; if $\theta^*> \underline{z}$, \KT{i.e, the lower bound of the current branch is tighter than the global one,} a \emph{locally valid} \KT{and tighter} cut may be derived
by using $\theta^*$ as the \KT{strengthening} bound. 

\section{Applications}
\label{section:applications}
In this section, we present the two applications for which we conduct a numerical study. 
A main difference between the two test-cases is in the computational complexity of the
associated recourse problems; indeed, this turns out to be an NP-hard problem in the first case and a polynomially 
solvable problem in the second one.

\subsection{Knapsack Fortification Game}
\label{section:KIF}

The first case-study we consider is the knapsack fortification game (KFG).
The classical 0-1 knapsack problem (KP) has been widely studied in the literature because of its practical and theoretical
relevance, and because it arises as a subproblem in many more complex problems (see,
e.g., \cite{KPP04li}). In this problem,
we are given a set $N$ of items, the $i$-th item having profit $d_i \ge 0$ and weight $a_i \ge 0$ and a knapsack having capacity  $b$.  
In the three-level fortification version of KP, each item $i$ has associated additional weights $f_i \ge 0$ and $g_i \ge 0$ for the fortification and 
interdiction levels, respectively; in addition, fortification and interdiction strategies are subject to some budget, denoted as 
$B_F >0$ and $B_I > 0$, respectively.
First the defender chooses a set of items to fortify against interdiction while respecting the fortification budget. 
Then, the attacker interdicts some of the non-fortified items 
within the interdiction budget. 
Lastly, the defender determines the recourse action, 
by solving a KP over non-interdicted items, i.e., by choosing a subset of non-interdicted items giving the maximum profit while not exceeding the knapsack capacity. 
The feasible regions of the first and the second levels are denoted by $W=\{w\in \{0,1\}^{n}: \sum_{i\in N} f_i w_{i}\leq B_F\}$ 
and $X(w) = X \cap \{x\in \mathbb{R}^n: x\leq 1-w \}$, respectively, where $X=\{x\in \{0,1\}^{n}:\sum_{i\in N} g_i x_i \leq B_I\}$. 
The tri-level game is thus formulated by the following program:

\begin{align}
z^*= \, \max_{w\in W}\, \min_{x \in X(w)}\,   \max_{y} & \sum_{i\in N}d_{i} y_{i} \\
&  \sum_{i \in N} a_i y_i\leq b \\
&  y_i \leq 1-x_i  \hspace*{1cm} \forall i \in N  \label{eq:KFG_interdiction}\\
&  y_{i} \in \{0,1\} \hspace*{1cm} \forall i \in N
\end{align}
Since the recourse problem satisfies the down-monotonicty assumption \citep{fischetti2019interdiction}, the interdiction constraints \eqref{eq:KFG_interdiction} can be replaced by penalty terms in the objective function with $M_i=d_i$ as explained in Remark \ref{remark_P1}. With necessary transformations for a defender with a maximization objective, the single level formulation of the KFG becomes
\begin{align}
\text{KFG} : \quad  z^*=  \, \max_{w \in W} \quad & \theta \\
& \theta \leq \Phi_I(\hat{x}) + \sum_{i\in N} d_{i}\hat{x}_{i}w_{i} \hspace*{1cm} \forall \hat{x} \in X .
\end{align}

\MS{We note that the KFG is $\Sigma^p_3$-hard already for the special case $d_i=a_i$, for all $i \in N$ as shown by \cite{nabli2020complexity}.}

\subsection{Shortest Path Fortification Game}
\label{section:SPIF}

The second application we consider is the shortest path fortification game (SPFG). 
In this game, also considered by \cite{lozano2017backward}, we are given a directed graph $G=(V,A)$ where $V$ and $A$ denote the set of nodes and arcs, respectively. The first level of the SPFG consists in selecting a set of arcs to protect from the interdiction by an attacker 
in order to minimize the length of the shortest $s-t$ path. 
Then, the attacker interdicts some of the non-fortified arcs, with the goal of maximizing the shortest path in the interdicted graph. 
At a third level, a recourse step is executed by solving a shortest path problem in the interdicted network.
The mathematical formulation of the game is given as
\begin{align}
z^*= \, \min_{w\in W}\, &\max_{x \in X(w)}\,  \min_{y}  \sum_{(i,j)\in A}(c_{ij} + d_{ij} x_{ij}) y_{ij} &&& \\
 \sum_{j:(i,j)\in A} y_{ij} -\sum_{j:(j,i)\in A} y_{ji} & = \begin{cases} 
      1 & i=s \\
      0 & i\in V\setminus \{s,t\} \\
      -1 & i=t 
   \end{cases} && \forall i \in V \\
  y_{ij} & \geq 0 &&  \forall (i,j)\in A,
\end{align}
where $c_{ij}$ denotes the nominal cost of each arc $(i,j) \in A$ and $d_{ij}$ represents the additional cost due to its interdiction. 
The feasible regions of the first and second level problems are described by cardinality constraints, i.e., $W=\{w\in \{0,1\}^{m}: \sum_{(i,j)\in A} w_{ij}\leq B_F\}$, $X=\{x\in \{0,1\}^{m}:\sum_{(i,j)\in A} x_{ij}\leq B_I\}$, and $X(w)=X\cap \{x\in \mathbb{R}^m: x\leq 1-w \}$, where $m$ denotes the the number of arcs.  The fortification game is reformulated by the following single level formulation following the steps described in Section \ref{section:FortCut}.
\begin{align}
\text{SPFG}: \quad
z^*= \, \min_{w\in W} \quad & \theta \\
& \theta \geq \Phi_I(\hat{x}) - \sum_{(i,j)\in A} d_{ij}\hat{x}_{ij}w_{ij} \hspace*{1cm} \forall \hat{x} \in X
\end{align}
where $X$ includes all feasible interdiction patterns. Notice that replacing 
$\Phi_I(\hat{x})$ with ${\sum_{(i,j)\in A}} (c_{ij} + d_{ij} \hat{x}_{ij}) \hat{y}_{ij}$,
where $\hat{y}$ is 
an optimal path under costs $c_{ij} + d_{ij} \hat{x}_{ij}$, gives the following fortification cut:
\begin{equation}
\theta \geq \sum_{(i,j)\in A} \Big( c_{ij} \hat{y}_{ij} + d_{ij}( \hat{y}_{ij}- w_{ij})\hat{x}_{ij} \Big).
\end{equation}
Let $(i,j)$ be an arc such that $\hat{x}_{ij}=1$ and $\hat{y}_{ij}=0$. Then, its contribution to the right-hand-side cost reduces to $-d_{ij}w_{ij}$. This means that if $(i,j)$ 
would be fortified (and thus not interdicted), the shortest path length would be reduced by at most $d_{ij}$.

\section{Computational Results}
\label{section:Results}
In this section we discuss the results of our computational study which has two main objectives.
First, to assess the computational performances of the proposed algorithms by comparison with the existing literature. Second, to
analyze the computational effectiveness of the proposed algorithmic enhancements and determine the limitations of the new approach.
We present computational results on two data sets from the literature for each of the two case-studies
we consider. 
%
Our algorithm has been implemented in C++, and makes use of IBM ILOG CPLEX 12.10 (in its default settings) as mixed-integer linear programming solver. 
All experiment are executed on a single thread of an Intel Xeon E5-2670v2 machine with 2.5 GHz processor and using a memory limit of 8~GB. Unless otherwise indicated, a time limit of one hour was imposed for each run.
Throughout this section we use the following notation to describe the components of our algorithm: bound based \KT{strengthening} (B), enumerative \KT{strengthening} (E), greedy integer separation (G), \KT{strengthening} of interdiction cuts (I), and any combination of these letters to denote the settings involving the corresponding methods.

\subsection{Results for the KFG}
In the following, we first describe problem specific implementation details of the B\&C for the KFG. Then, we discuss the results of our experiments on two data sets derived from publicly available instances for the Knapsack Interdiction Problem.

\subsubsection{Implementation Details for the KFG}

\KT{While solving (SEP)} for the KFG, separation of integer solutions accounts to solving a KP. Though specific efficient codes are available for
the exact solution of KP instances (see, e.g., procedure {\tt combo} in \cite{MPT99}), we solve this problem as a general MIP.
We perform 
the separation of fractional solutions
using the classical greedy algorithm for the knapsack problem; notice that, 
an exact separation 
of fractional solutions is not needed for the correctness of the algorithm.
Similarly, in Step~\ref{alg:gi:recourse} of Algorithm \ref{alg:GreedyInterdiction}, the recourse is solved greedily. 
Preliminary experiments showed that the \KT{strengthening} of interdiction cuts, 
though producing violated cuts, leads to  attacker solutions of poor quality, which reduces the strength of the fortification cuts. 
Besides, we did not observe any significant reduction in the solution time of (SEP) for a given leader solution, 
%
thus, for the KFG we focus on the algorithm configurations without component I.

As to the enumerative \KT{strengthening} procedure described in Section \ref{section:Lift_enum},  instead of solving the recourse problem exactly (i.e., solving the KP as an MIP) we only solve its LP relaxation, cf.\ Step~\ref{alg:el:recourse} of Algorithm~\ref{alg:EnumLift}.
In this way we obtain, with limited computational effort, a valid dual bound that can be used for \KT{strengthening} according to the observation in Remark \ref{remark_liftEnum}.
As described in Section \ref{section:lifting details}, a set $P$ satisfying the conditions for \KT{strengthening} can be removed from consideration if it includes items that cannot be used together in some feasible recourse solution.
To eliminate such $P$, we simply check the knapsack constraint for the recourse solution $y$ having $y_i=1$, $\forall i\in P$ and $y_i=0$, $\forall i\notin P$. If it is violated, we move to the next iteration. \KT{During the solution of any instance, if 10 successive trials of enumerative strengthening do not produce a strictly tighter cut, we do not employ this feature for the newly generated cuts.}

\subsubsection{TRS Knapsack Interdiction Dataset}
In this section, we consider the 150 instances introduced by \cite{tang2016class} for the 
Knapsack Interdiction Problem, and add a 
fortification level involving a cardinality constraint on the number of items that can be fortified. 
We use the fortification budget levels $B_F =3$ and $B_F=5$, thus producing a set of 300 benchmark instances, denoted as TRS. 
We solved each instance with several algorithm settings, including components B, E, and G 
in an incremental way. 
In Figure \ref{fig:TRS_rootgap}, we show the cummulative distribution of root gaps and solution times for the TRS instances. 
It can be seen that bound based \KT{strengthening slightly reduces} the solution times although it does not produce
consistent improvements in terms of root gaps. The next component, enumerative \KT{strengthening}, improves both measures, while its effect on the root gaps is more evident. It turns out that enhancing the integer separation by the greedy algorithm improves the performance as well,
in particular for what concerns the maximum time needed for solving the instances, which is 
significantly smaller with the setting BEG than with the other settings considered. All 300 instances are solved optimally within 100 seconds under all four settings. With BEG, the maximum solution time is \KT{18} seconds (cf.\ Figure \ref{fig:TRS_rootgap}).

\begin{figure}[h!tb]
\includegraphics[width=0.5\textwidth]{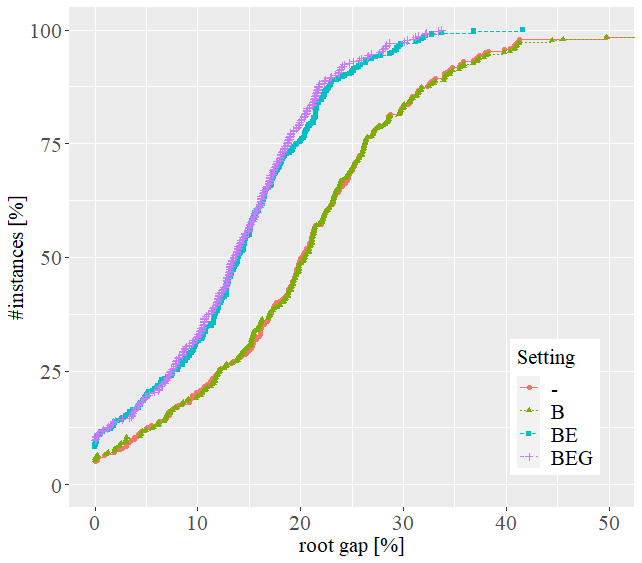}%
\includegraphics[width=0.5\textwidth]{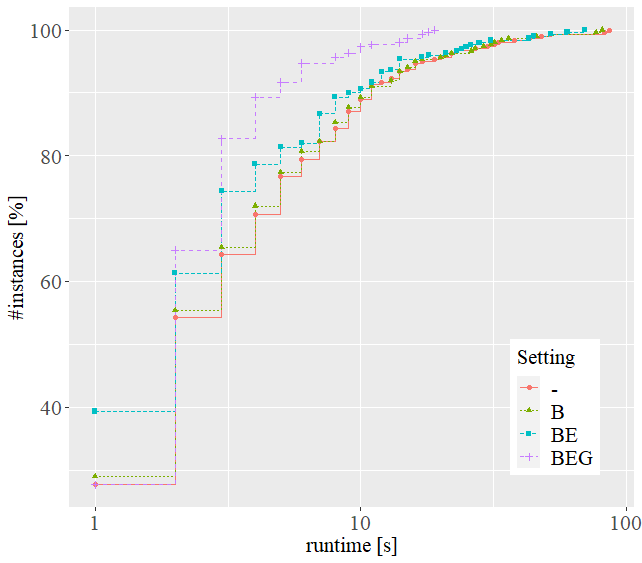}
\caption{Root gaps and solution times of TRS data set instances.\label{fig:TRS_rootgap}}
\end{figure}

\subsubsection{CCLW Knapsack Interdiction Dataset}

Our second benchmark for the KFG is derived from the instances proposed by \cite{caprara2016bilevel} 
for the Knapsack Interdiction Problem. In this case too, we introduce an additional fortification level, and impose 
a cardinality constraint for the leader. For interdiction costs and budgets we used the data available in the original instances.
The number of items takes value in $\{35, 40, 45, 50, 55\}$ and for each size there are ten instances in which the interdiction budget increases 
with the $ID$ of the instance, $1\leq ID \leq 10$. We solved these 50 instances for $B_F \in \{3, 5\}$, 
thus producing a set of 100 benchmark instances denoted as CCLW. 

Figure \ref{fig:CCLW_plot} shows the cumulative distribution of root gaps and solution times for the instances in this benchmark set. 
It can be seen that bound based \KT{strengthening} does not have an effect in both measures. 
This is mostly due to the fact that we are able to produce only a few \KT{strengthened} cuts, as the value of $\Phi(\hat{x})-\underline{z}$ 
is usually larger than $d_i$ in our experiments. 
Enumerative \KT{strengthening} does not contribute to decreasing the solution times because of its time complexity, which depends on the interdiction budget; however this strategy decreases the root gap on average. Note that the root gap reduction is larger for smaller budget levels. 
We also observe smaller number of B\&C nodes when enumerative \KT{strengthening} is activated, although we do not report details about this kind of 
measure here.
Similar to the results on the TRS data set, the most effective component \KT{regarding the solution time} is greedy integer separation ($G$). We observe in our experiments that Algorithm \ref{alg:GreedyInterdiction} is able to find good interdiction strategies quickly, which makes the separation process more efficient compared to the default setting where we stop solving (SEP) once a violated fortification is found. We are able to solve 95 out of 100 instances optimally in one hour when using setting BEG.

\begin{figure}[h!tb]
\includegraphics[width=0.5\textwidth]{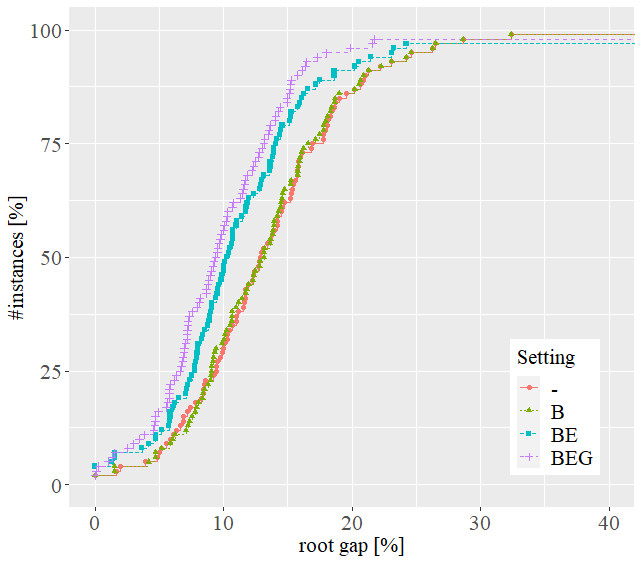}%
\includegraphics[width=0.5\textwidth]{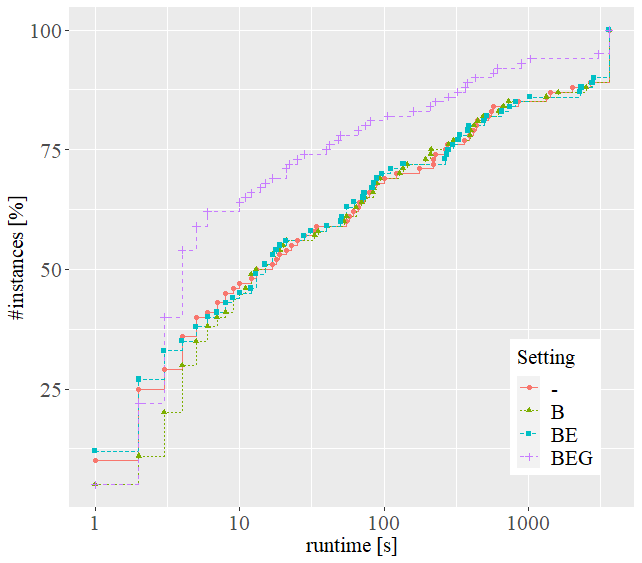}
\caption{Root gaps and solution times of CCLW data set instances.\label{fig:CCLW_plot}}
\end{figure}

\subsection{Results for the SPFG}

In the following, we first describe problem specific implementation details of the B\&C algorithm for the SPFG. Then, we present the results of our experiments on two existing data sets from the literature.

\subsubsection{Implementation Details for the SPFG}
The recourse problem, which is the shortest path problem, is solved via Dijkstra's algorithm \citep{dijkstra1959note} with a priority queue implementation, wherever needed, e.g., in
Algorithm~\ref{alg:GreedyInterdiction} and Algorithm~\ref{alg:EnumLift}. 
While solving (SEP) to obtain a fortification cut, both integer and fractional solutions are separated exactly.
Separation is carried out in $O(|A|\log |V|)$ time. 
In addition to the lower cutoff and solution limit strategies, interdiction cut \KT{strengthening} is used to speed up the solution of (SEP). 
Recall that any value $l>\theta^*$ can be used to lift the interdiction cut. Since all the problem parameters are integers in out data sets, we choose $l=\lceil\theta^*+\epsilon \rceil$. 
Remind that fortifying two assets (i.e., arcs) at the same time can be more effective than the sum of their individual effects only if the
two arcs can be used together in a recourse solution (see Section \ref{section:lifting details}). Accordingly, while executing 
Algorithm~\ref{alg:EnumLift}, we try to not consider all sets $P$ including pairs of arcs that cannot appear together in an $s-t$ path.
Since checking this exactly would be time consuming,
we only check if any two arcs in $P$ have the same head or tail and skip the current iteration if there is such an arc pair, which indicates that an $s-t$ path cannot include both of these arcs. 

\subsubsection{Directed Grid Networks}

\begin{table}[b!]
	\centering \small   
 \caption{The sizes of the grid networks.\label{tab:GridNetSizes}}
    \begin{tabular}{crr}
\toprule
Instance & \multicolumn{1}{c}{Nodes} & \multicolumn{1}{c}{Arcs} \\
 \midrule
 10x10 & 102   & 416 \\
    20x20 & 402   & 1,826 \\
    30x30 & 902   & 4,236 \\
    40x40 & 1,602 & 7,646 \\
    50x50 & 2,502 & 12,056 \\
   60x60 & 3,602 & 17,466 \\
    \bottomrule
    \end{tabular}
\end{table}%

For the SPFG, the first data set we use is the directed grid network instance set generated by \cite{lozano2017backward} following the topology used by \cite{israeli2002shortest} and \cite{cappanera2011optimal}. These instances \MS{and the implementation of the solution algorithm proposed in \cite{lozano2017backward}} are available at \url{https://doi.org/10.1287/ijoc.2016.0721}. In these networks, in addition to a source and a sink node, there are $m\times n$ nodes forming a grid with $m$ rows and $n$ columns. The source (sink) node is linked with each of the grid nodes in the first (last) column. The sizes of the networks are shown in Table \ref{tab:GridNetSizes}. We use the cost ($c_{ij}$) and delay ($d_{ij}$) values 
provided with the networks, which were generated randomly in $[1,c_{max}]$ and $[1,d_{max}]$, respectively, using six different $c_{max}$ and $d_{max}$ combinations. \KT{For each network size, ten instances with different arc costs/delays are generated.} Each instance is solved for six different fortification and interdiction budget level configurations $(B_F,B_I)$ also used in \citet{lozano2017backward}.

In Figure \ref{fig:Grid_plot} the cumulative distributions of root gaps and solution times 
are plotted for three settings: 
the basic one ``-", the best performing setting for KFG instances (BEG), and a variant additionally considering interdiction cut \KT{strengthening} (IBEG). 
The plot on the left hand side shows that the root gaps of variant BEG are substantially smaller than those of the basic setting.
It is also clear from the plots that the \KT{strengthening} of interdiction cuts using $l=\lceil\theta^*+\epsilon \rceil$ produces a remarkable improvement. 
We observed that,  
using this strategy the number of B\&C nodes increased with respect to BEG, since attacker solutions with smaller objective values 
(and, in turn, weaker fortification cuts) are produced. Nevertheless, in this case, the separation problem can be solved extremely efficiently, which 
overall leads to improved performances.

\begin{figure}[h!tb]
\includegraphics[width=0.5\textwidth]{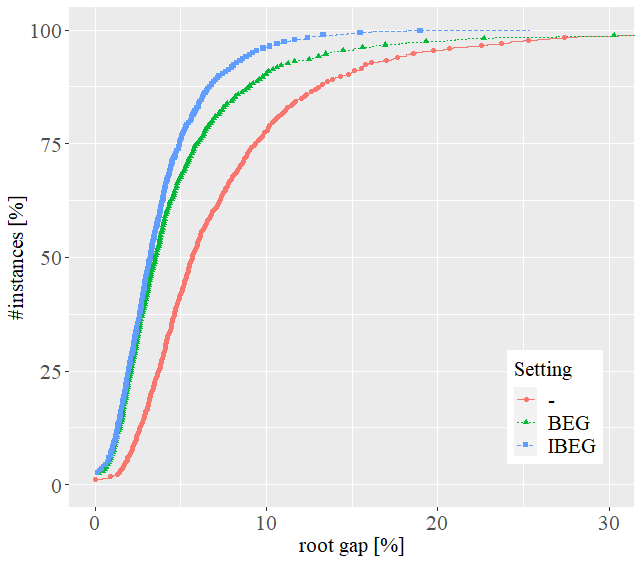}%
\includegraphics[width=0.5\textwidth]{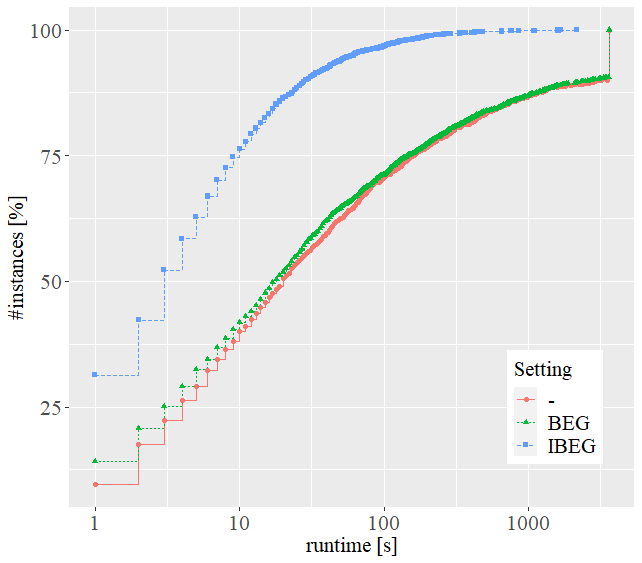}
\caption{Root gaps and solution times of Grid networks data set.\label{fig:Grid_plot}}
\end{figure}

Tables \ref{tab:GridRuntimes1} and \ref{tab:GridRuntimes2} report
the average and maximum solution times required by our algorithm with setting IBEG. These figures are compared with the results \MS{obtained by the solution algorithm of \cite{lozano2017backward}, denoted as LS. Since 
the implementation of this algorithm is available online, we were able to run it on the same hardware as our algorithm. The algorithm was coded in Java \ML{and} we ran it with Gurobi 9.2 as MIP solver.} 
In both tables, rows are associated to a network size and a configuration for the fortification and interdiction budget levels,
whereas the three blocks of columns refer to different cost-delay configurations \KT{$(c_{max}\mbox{-}d_{max})$, i.e., (10-5), (10-10), and (10-20) in Table  \ref{tab:GridRuntimes1}; (100-50), (100-100), and (100-200) in Table  \ref{tab:GridRuntimes2}}.

\begin{table}[h!tb]
  \centering
 \caption{SPFG Solution Times for the Directed Grid Networks with Setting IBEG \MM{and with algorithm LS} (Part 1).\label{tab:GridRuntimes1}}
     \footnotesize
    \begin{tabular}{cccrrrr|rrrr|rrrr}
\toprule         &       &       & \multicolumn{4}{c|}{(10-5)}  & \multicolumn{4}{c|}{(10-10)} & \multicolumn{4}{c}{(10-20)} \\
\cmidrule{4-15}           &       &       & \multicolumn{2}{c}{\AlgName} & \multicolumn{2}{c|}{LS} & \multicolumn{2}{c}{\AlgName} & \multicolumn{2}{c|}{LS} & \multicolumn{2}{c}{\AlgName} & \multicolumn{2}{c}{LS} \\
\cmidrule{4-15}           & $B_F$     & $B_I$     & Avg   & Max   & Avg   & Max   & Avg   & Max   & Avg   & Max   & Avg   & Max   & Avg   & Max \\
    \midrule
    \multirow{6}[2]{*}{10x10} & 3     & 3     & 0.1   & 0.1   & 0.1   & 0.2   & 0.1   & 0.1   & 0.1   & 0.2   & 0.1   & 0.1   & 0.1   & 0.2 \\
          & 4     & 3     & 0.1   & 0.1   & 0.1   & 0.3   & 0.1   & 0.2   & 0.2   & 0.2   & 0.1   & 0.1   & 0.2   & 0.2 \\
          & 3     & 4     & 0.1   & 0.2   & 0.2   & 0.3   & 0.2   & 0.2   & 0.3   & 0.4   & 0.2   & 0.2   & 0.3   & 0.6 \\
          & 5     & 4     & 0.2   & 0.2   & 0.3   & 0.5   & 0.2   & 0.3   & 0.4   & 0.5   & 0.2   & 0.4   & 0.4   & 0.6 \\
          & 4     & 5     & 0.2   & 0.3   & 0.4   & 0.7   & 0.2   & 0.3   & 0.6   & 1.0   & 0.3   & 0.4   & 0.7   & 1.5 \\
          & 7     & 5     & 0.3   & 0.6   & 1.1   & 1.7   & 0.4   & 0.6   & 1.1   & 1.5   & 0.6   & 1.0   & 1.3   & 1.9 \\
    \midrule
    \multirow{6}[2]{*}{20x20} & 3     & 3     & 0.3   & 0.5   & 0.4   & 1.0   & 0.3   & 0.6   & 0.6   & 2.1   & 0.3   & 0.8   & 0.6   & 2.6 \\
          & 4     & 3     & 0.3   & 0.5   & 0.5   & 1.0   & 0.4   & 0.8   & 0.7   & 2.5   & 0.4   & 0.9   & 0.7   & 2.7 \\
          & 3     & 4     & 0.4   & 1.0   & 0.6   & 1.3   & 0.4   & 0.8   & 1.0   & 2.5   & 0.5   & 0.8   & 1.2   & 5.2 \\
          & 5     & 4     & 0.6   & 1.4   & 0.9   & 1.8   & 0.8   & 1.8   & 1.5   & 5.1   & 0.7   & 1.4   & 1.7   & 7.3 \\
          & 4     & 5     & 0.8   & 1.2   & 2.1   & 5.6   & 0.9   & 1.6   & 2.7   & 9.9   & 0.9   & 1.4   & 3.0   & 11.4 \\
          & 7     & 5     & 1.7   & 3.3   & 3.3   & 5.9   & 2.4   & 3.4   & 5.1   & 16.7  & 2.6   & 3.9   & 7.1   & 31.5 \\
    \midrule
    \multirow{6}[2]{*}{30x30} & 3     & 3     & 0.6   & 1.0   & 0.9   & 1.2   & 0.6   & 1.2   & 1.4   & 3.8   & 0.6   & 1.0   & 3.6   & 19.0 \\
          & 4     & 3     & 0.6   & 0.9   & 1.1   & 1.8   & 0.8   & 1.3   & 1.8   & 4.6   & 0.8   & 1.2   & 3.9   & 19.6 \\
          & 3     & 4     & 1.2   & 1.8   & 1.9   & 2.7   & 1.4   & 2.9   & 5.6   & 27.7  & 1.4   & 3.1   & 8.0   & 35.1 \\
          & 5     & 4     & 1.9   & 3.4   & 2.7   & 3.9   & 2.1   & 3.1   & 5.8   & 16.4  & 2.0   & 2.9   & 9.2   & 35.9 \\
          & 4     & 5     & 2.7   & 4.9   & 6.4   & 9.9   & 2.9   & 5.9   & 19.2  & 110.5 & 3.4   & 8.7   & 21.5  & 108.7 \\
          & 7     & 5     & 5.6   & 9.7   & 12.3  & 24.5  & 7.2   & 13.8  & 29.2  & 146.2 & 7.6   & 12.8  & 28.0  & 115.1 \\
    \midrule
    \multirow{6}[2]{*}{40x40} & 3     & 3     & 1.7   & 3.5   & 2.7   & 6.0   & 1.8   & 4.3   & 5.4   & 26.9  & 1.8   & 4.3   & 7.8   & 43.3 \\
          & 4     & 3     & 1.9   & 3.0   & 3.6   & 7.3   & 2.4   & 4.6   & 6.8   & 32.9  & 2.2   & 3.8   & 9.0   & 45.6 \\
          & 3     & 4     & 2.7   & 5.1   & 3.9   & 6.3   & 3.3   & 8.2   & 37.4  & 267.9 & 3.3   & 9.3   & 63.6  & 540.0 \\
          & 5     & 4     & 4.9   & 8.4   & 9.1   & 30.5  & 5.3   & 12.3  & 42.2  & 260.8 & 5.9   & 14.4  & 69.9  & 549.6 \\
          & 4     & 5     & 6.0   & 11.3  & 13.7  & 33.4  & 10.5  & 38.8  & 247.3 & 1943.5 & 9.5   & 30.9  & 377.7 & 2786.6 \\
          & 7     & 5     & 15.6  & 32.4  & 26.5  & 76.0  & 22.9  & 66.8  & 486.4 & 3514.0 & 27.9  & 124.1 & 516.7 & 4106.4 \\
    \midrule
    \multirow{6}[2]{*}{50x50} & 3     & 3     & 2.1   & 3.4   & 4.4   & 11.1  & 2.1   & 4.2   & 10.7  & 38.0  & 2.2   & 4.7   & 13.3  & 41.6 \\
          & 4     & 3     & 2.8   & 6.8   & 6.6   & 25.2  & 2.8   & 5.0   & 19.7  & 82.2  & 2.7   & 5.3   & 14.4  & 45.4 \\
          & 3     & 4     & 3.3   & 8.1   & 13.1  & 83.7  & 5.0   & 11.8  & 53.4  & 240.4 & 4.8   & 17.0  & 52.5  & 371.5 \\
          & 5     & 4     & 5.9   & 9.6   & 15.8  & 87.6  & 7.6   & 20.2  & 112.9 & 627.1 & 8.5   & 28.1  & 61.1  & 416.2 \\
          & 4     & 5     & 9.5   & 24.0  & 38.6  & 276.0 & 17.7  & 75.9  & 219.6 & 1068.9 & 14.6  & 67.5  & 112.6 & 858.0 \\
          & 7     & 5     & 24.8  & 45.9  & 77.9  & 518.4 & 33.5  & 124.3 & 407.3 & 2325.3 & 38.1  & 149.2 & 102.3 & 624.7 \\
    \midrule
    \multirow{6}[2]{*}{60x60} & 3     & 3     & 4.0   & 8.1   & 11.4  & 30.4  & 4.4   & 7.4   & 21.8  & 65.6  & 4.1   & 6.4   & 32.0  & 73.7 \\
          & 4     & 3     & 5.4   & 13.1  & 12.6  & 33.9  & 5.9   & 13.3  & 32.7  & 143.7 & 6.1   & 18.2  & 38.0  & 95.8 \\
          & 3     & 4     & 7.1   & 13.9  & 15.1  & 28.2  & 10.5  & 18.1  & 69.6  & 299.4 & 11.1  & 23.3  & 95.0  & 426.7 \\
          & 5     & 4     & 12.7  & 27.9  & 21.2  & 54.8  & 16.7  & 30.6  & 134.4 & 905.8 & 16.6  & 28.8  & 115.8 & 575.2 \\
          & 4     & 5     & 18.8  & 26.3  & 40.5  & 76.6  & 35.5  & 107.3 & 285.4 & 915.7 & 30.4  & 75.1  & 265.7 & 1056.2 \\
          & 7     & 5     & 45.7  & 80.6  & 108.6 & 351.8 & 87.0  & 185.0 & 624.8 & 3440.4 & 74.2  & 127.9 & 474.1 & 2155.0 \\
    \bottomrule
    \end{tabular}%
  \label{tab:addlabel}%
\end{table}%

Table \ref{tab:GridRuntimes1} contains the results for smallest three cost-delay configurations, and shows that \AlgName{} yields \KT{significantly} smaller 
average and maximum solution times \KT{for most of the instances}.
In \KT{104 (107)} out of 108 configurations, \AlgName{} produces a strictly smaller average (maximum) solution time. 
For \KT{many instances of larger sizes} our algorithm is faster than LS by one order of magnitude.
\KT{Overall}, the performance improvement of \AlgName{} over LS becomes more pronounced as the $c_{max}$ and $d_{max}$ values increase.

\begin{table}[h!tb]
\centering
\caption{SPFG Solution Times for the Directed Grid Networks with Setting IBEG \MM{and with algorithm LS} (Part 2).\label{tab:GridRuntimes2}}
\footnotesize
    \begin{tabular}{cccrrrr|rrrr|rrrr}
\midrule         &       &       & \multicolumn{4}{c|}{(100-50)}  & \multicolumn{4}{c|}{(100-100)} & \multicolumn{4}{c}{(100-200)} \\
\cmidrule{4-15}           &       &       & \multicolumn{2}{c}{\AlgName} & \multicolumn{2}{c|}{LS} & \multicolumn{2}{c}{\AlgName} & \multicolumn{2}{c|}{LS} & \multicolumn{2}{c}{\AlgName} & \multicolumn{2}{c}{LS} \\
\cmidrule{4-15}           & $B_F$     & $B_I$     & Avg   & Max   & Avg   & Max   & Avg   & Max   & Avg   & Max   & Avg   & Max   & Avg   & Max \\
    \midrule
    \multirow{6}[2]{*}{10x10} & 3     & 3     & 0.1   & 0.2   & 0.1   & 0.2   & 0.3   & 0.4   & 0.2   & 0.3   & 0.1   & 0.2   & 0.2   & 0.2 \\
          & 4     & 3     & 0.2   & 0.2   & 0.2   & 0.2   & 0.2   & 0.2   & 0.2   & 0.3   & 0.2   & 0.3   & 0.2   & 0.3 \\
          & 3     & 4     & 0.3   & 0.4   & 0.3   & 0.5   & 0.2   & 0.4   & 0.4   & 0.6   & 0.3   & 0.4   & 0.4   & 0.9 \\
          & 5     & 4     & 0.3   & 0.4   & 0.4   & 0.5   & 0.3   & 0.7   & 0.5   & 0.8   & 0.3   & 0.5   & 0.5   & 0.9 \\
          & 4     & 5     & 0.4   & 0.5   & 0.8   & 1.3   & 0.5   & 0.7   & 1.0   & 2.0   & 0.6   & 0.9   & 1.1   & 2.4 \\
          & 7     & 5     & 0.5   & 1.0   & 1.3   & 2.4   & 0.8   & 1.2   & 1.7   & 3.3   & 0.9   & 1.7   & 1.9   & 3.4 \\
    \midrule
    \multirow{6}[2]{*}{20x20} & 3     & 3     & 0.5   & 0.8   & 0.7   & 1.9   & 0.5   & 0.7   & 0.7   & 1.1   & 0.5   & 0.9   & 0.7   & 1.5 \\
          & 4     & 3     & 0.6   & 0.9   & 0.9   & 1.7   & 0.6   & 1.1   & 0.8   & 1.2   & 0.7   & 1.2   & 0.8   & 1.7 \\
          & 3     & 4     & 1.0   & 1.5   & 2.1   & 5.4   & 1.0   & 1.3   & 1.8   & 4.4   & 0.9   & 1.6   & 1.8   & 4.6 \\
          & 5     & 4     & 1.7   & 3.1   & 2.6   & 4.8   & 1.8   & 3.0   & 2.4   & 6.1   & 1.7   & 3.3   & 2.4   & 5.7 \\
          & 4     & 5     & 2.4   & 3.4   & 5.8   & 14.5  & 2.7   & 4.7   & 7.7   & 22.6  & 2.4   & 4.5   & 7.8   & 25.2 \\
          & 7     & 5     & 6.3   & 9.6   & 7.7   & 12.3  & 5.9   & 8.9   & 9.4   & 22.9  & 5.5   & 8.7   & 9.5   & 28.7 \\
    \midrule
    \multirow{6}[2]{*}{30x30} & 3     & 3     & 1.2   & 2.3   & 2.9   & 17.5  & 1.5   & 3.1   & 1.9   & 4.0   & 1.3   & 2.2   & 2.4   & 6.0 \\
          & 4     & 3     & 1.4   & 2.9   & 2.1   & 9.5   & 1.7   & 2.9   & 2.7   & 6.8   & 1.6   & 2.5   & 2.7   & 6.8 \\
          & 3     & 4     & 2.5   & 5.1   & 4.6   & 13.7  & 2.3   & 3.3   & 5.1   & 14.7  & 2.4   & 3.4   & 4.8   & 10.8 \\
          & 5     & 4     & 5.1   & 12.7  & 7.3   & 22.3  & 4.0   & 7.2   & 7.6   & 24.0  & 3.9   & 6.1   & 7.1   & 22.6 \\
          & 4     & 5     & 6.7   & 10.0  & 19.0  & 97.4  & 6.0   & 10.1  & 22.3  & 52.6  & 6.3   & 10.8  & 19.6  & 75.3 \\
          & 7     & 5     & 19.8  & 47.5  & 29.9  & 89.2  & 16.1  & 23.4  & 49.1  & 218.1 & 16.1  & 23.8  & 42.6  & 175.1 \\
    \midrule
    \multirow{6}[2]{*}{40x40} & 3     & 3     & 2.3   & 4.3   & 3.4   & 6.7   & 2.8   & 4.5   & 6.9   & 31.0  & 2.4   & 3.6   & 6.1   & 18.9 \\
          & 4     & 3     & 3.0   & 3.9   & 4.1   & 6.6   & 3.2   & 7.0   & 9.3   & 53.3  & 3.2   & 7.3   & 7.5   & 23.7 \\
          & 3     & 4     & 6.6   & 11.8  & 11.3  & 37.0  & 6.1   & 13.9  & 14.9  & 40.1  & 4.9   & 9.0   & 13.0  & 21.0 \\
          & 5     & 4     & 11.0  & 22.7  & 16.3  & 50.3  & 10.0  & 18.6  & 27.2  & 137.6 & 8.6   & 14.4  & 22.4  & 85.4 \\
          & 4     & 5     & 21.3  & 49.9  & 32.5  & 78.3  & 22.9  & 92.3  & 75.2  & 277.6 & 13.7  & 33.1  & 44.9  & 138.1 \\
          & 7     & 5     & 56.8  & 95.1  & 53.0  & 128.1 & 57.9  & 144.5 & 157.2 & 934.0 & 50.4  & 174.6 & 114.4 & 705.3 \\
    \midrule
    \multirow{6}[2]{*}{50x50} & 3     & 3     & 4.3   & 10.1  & 10.0  & 33.0  & 5.6   & 18.4  & 36.4  & 234.9 & 5.2   & 19.4  & 40.5  & 182.0 \\
          & 4     & 3     & 5.6   & 12.4  & 10.1  & 27.5  & 6.0   & 17.3  & 49.5  & 359.1 & 6.2   & 23.3  & 40.8  & 173.8 \\
          & 3     & 4     & 12.9  & 37.0  & 28.7  & 98.3  & 15.0  & 42.9  & 108.8 & 763.7 & 9.5   & 25.8  & 76.2  & 268.2 \\
          & 5     & 4     & 27.3  & 89.4  & 61.1  & 204.3 & 26.6  & 102.6 & 417.3 & 3678.8 & 20.0  & 52.2  & 104.7 & 496.7 \\
          & 4     & 5     & 57.3  & 206.9 & 120.4 & 494.4 & 69.4  & 251.0 & 253.2 & 1436.7 & 58.5  & 227.9 & 309.5 & 1545.9 \\
          & 7     & 5     & 281.0 & 1097.6 & 238.5 & 1099.6 & 135.7 & 402.5 & 912.6 & 7746.5 & 93.4  & 334.8 & 469.2 & 3025.0 \\
    \midrule
    \multirow{6}[2]{*}{60x60} & 3     & 3     & 8.2   & 25.0  & 14.0  & 40.9  & 9.2   & 28.2  & 52.1  & 338.1 & 8.3   & 22.0  & 62.3  & 250.7 \\
          & 4     & 3     & 11.4  & 24.4  & 17.6  & 58.2  & 11.4  & 34.7  & 76.2  & 552.1 & 12.9  & 41.9  & 68.4  & 297.4 \\
          & 3     & 4     & 21.2  & 33.8  & 44.9  & 92.0  & 25.7  & 104.6 & 260.6 & 1790.4 & 23.2  & 99.1  & 180.7 & 866.1 \\
          & 5     & 4     & 45.5  & 169.6 & 64.9  & 194.8 & 52.7  & 246.0 & 254.0 & 1726.6 & 33.1  & 110.1 & 202.4 & 1029.8 \\
          & 4     & 5     & 92.5  & 365.0 & 148.2 & 391.5 & 167.4 & 854.4 & 1281.5 & 7442.6 & 222.9 & 1677.1 & 1120.4 & 5620.4 \\
          & 7     & 5     & 271.8 & 761.7 & 329.8 & 1325.5 & 368.6 & 1608.8 & 1256.7 & 6658.2 & 364.7 & 2149.1 & 1493.8 & 8691.8 \\
    \bottomrule
    \end{tabular}%
  \label{tab:addlabel}%
\end{table}%

Table \ref{tab:GridRuntimes2} gives results for the largest cost-delay configurations. For these instances, \KT{again \AlgName{} 
outperforms LS for most of the instances}. Here, in \KT{100 (103)} out of 108 configurations, \AlgName{} produces a strictly smaller average (maximum) solution time.
As the instances get more difficult, \AlgName{} becomes more advantageous. For the most difficult group of instances with 
\KT{$n=60$ and $(c_{max}-d_{max})$ is (100-100) or (100-200), the average solution times with \AlgName{} are 80\% and 78\% smaller, respectively,  than their LS counterparts.}
Moreover, \AlgName{} solves all instances within the time limit of one hour.

\subsubsection{Real Road Networks}

\begin{table}[h!tb] 
\centering
\caption{Results for instance set \texttt{road network}.\label{tab:RoadNetTimes}}
\small
	\begin{tabular}{cccccrrr|rrr}
		\toprule
		\multicolumn{5}{l|}{} &\multicolumn{3}{c|}{\AlgName} &\multicolumn{3}{c}{LS}  \\
\multicolumn{1}{l}{Instance} & \multicolumn{1}{c}{Nodes} & \multicolumn{1}{c}{Arcs} & \multicolumn{1}{l}{$B_F$} & \multicolumn{1}{l|}{$B_I$} & \multicolumn{1}{l}{Avg(s.)} & \multicolumn{1}{l}{Max(s.)} & \multicolumn{1}{l|}{Solved} & \multicolumn{1}{l}{Avg(s.)} & \multicolumn{1}{l}{Max(s.)} & \multicolumn{1}{l}{Solved} 		 \\
\midrule
    \multicolumn{1}{l}{DC} & \multicolumn{1}{r}{9559} & \multicolumn{1}{r}{39377} & 3     & \multicolumn{1}{c|}{3} & 9.5   & 15.1  & \multicolumn{1}{r|}{9} & 35.0  & 99.2  & 9 \\
          &       &       & 4     & \multicolumn{1}{c|}{3} & 14.2  & 29.4  & \multicolumn{1}{r|}{9} & 39.9  & 101.8 & 9 \\
          &       &       & 3     & \multicolumn{1}{c|}{4} & 18.1  & 32.2  & \multicolumn{1}{r|}{9} & 94.2  & 453.7 & 9 \\
          &       &       & 5     & \multicolumn{1}{c|}{4} & 45.3  & 68.4  & \multicolumn{1}{r|}{9} & 104.0 & 426.5 & 9 \\
          &       &       & 4     & \multicolumn{1}{c|}{5} & 51.5  & 86.3  & \multicolumn{1}{r|}{9} & 523.6 & 2959.7 & 9 \\
          &       &       & 7     & \multicolumn{1}{c|}{5} & 123.7 & 230.7 & \multicolumn{1}{r|}{9} & 608.4 & 2126.0 & 9 \\
    \midrule
    \multicolumn{1}{l}{RI} & \multicolumn{1}{r}{53658} & \multicolumn{1}{r}{192084} & 3     & \multicolumn{1}{c|}{3} & 88.2  & 234.2 & \multicolumn{1}{r|}{9} & 137.7 & 588.3 & 9 \\
          &       &       & 4     & \multicolumn{1}{c|}{3} & 164.3 & 289.3 & \multicolumn{1}{r|}{9} & 144.2 & 570.1 & 9 \\
          &       &       & 3     & \multicolumn{1}{c|}{4} & 285.7 & 960.1 & \multicolumn{1}{r|}{9} & 463.6 & 3031.8 & 9 \\
          &       &       & 5     & \multicolumn{1}{c|}{4} & 389.0 & 1432.8 & \multicolumn{1}{r|}{9} & 465.4 & 2621.2 & 9 \\
          &       &       & 4     & \multicolumn{1}{c|}{5} & 1084.4 & 5766.4 & \multicolumn{1}{r|}{9} & 1896.3 & 13758.4 & 9 \\
          &       &       & 7     & \multicolumn{1}{c|}{5} & 1565.7 & 6042.8 & \multicolumn{1}{r|}{9} & 2092.1 & TL & 8 \\
          \midrule
    \multicolumn{1}{l}{NJ} & \multicolumn{1}{r}{330386} & \multicolumn{1}{r}{1202458} & 3     & \multicolumn{1}{c|}{3} & 858.3 & 1997.3 & \multicolumn{1}{r|}{9} & 988.5 & 5442.5 & 9 \\
          &       &       & 4     & \multicolumn{1}{c|}{3} & 1078.8 & 2906.5 & \multicolumn{1}{r|}{9} & 2057.4 & TL & 8 \\
          &       &       & 3     & \multicolumn{1}{c|}{4} & 1576.3 & 2849.0 & \multicolumn{1}{r|}{9} & 2217.3 & TL & 8 \\
          &       &       & 5     & \multicolumn{1}{c|}{4} & 2968.8 & 8052.1 & \multicolumn{1}{r|}{9} & 2557.4 & TL & 8 \\
          &       &       & 4     & \multicolumn{1}{c|}{5} & 5695.8 & TL & \multicolumn{1}{r|}{8} & 2702.0 & TL & 8 \\
          &       &       & 7     & \multicolumn{1}{c|}{5} & 10042.3 & TL & \multicolumn{1}{r|}{7} & 3885.4 & TL & 8 \\
    \midrule
    \multicolumn{5}{c}{Total}             &       &       & 159   &       &       & 156 \\
    \bottomrule
    \end{tabular}%
  \label{tab:addlabel}%
\end{table}%

For the second part of the SPFG experiments, we considered the real road network data sets of the cities Washington (DC),
Rhode Island (RI), and New Jersey (NJ), used by \cite{raith2009comparison} and \cite{lozano2017backward}. While the original undirected data sets are available at \url{http://www.diag.uniroma1.it//~challenge9/data/tiger/}, we used the data sets provided by \cite{lozano2017backward} where each edge is replaced by two arcs to get directed networks and in which connectivity is enforced by additional high-cost arcs. The costs of the arcs are equal to the distances and the delay is 10\,000 for each arc. On each of the three networks, the SPFG is solved for nine distinct $s-t$ pairs and using the six fortification and interdiction budget combinations used for the grid networks, resulting in 162 instances. As these are the most challenging instances considered in our study, the  time limit is set to four hours for each run. 
Table \ref{tab:RoadNetTimes} reports the sizes of the networks as well as average and maximum solution times over the nine instances with the same size and budget combination for IBEG and LS, as reported in \citet{lozano2017backward}.
The table shows that our default setting IBEG is able to solve 159 over 162 instances optimally within the time limit. {Notice that our B\&C algorithm solves \KT{three} instances more than the sampling based algorithm by \cite{lozano2017backward}, \KT{which was run on the same machine as \AlgName{} and produced significantly smaller solution times than the ones reported in their paper}. In addition, when both algorithms solve all the instances, IBEG is significantly faster than LS: the average computing time and the maximum computing times are 
below \KT{48\% and 72\% on the average}, respectively, of their counterpart for LS.
}

\section{Conclusion}
\label{section:conclusions}
In this study we address fortification games (FGs), i.e., defender-attacker-defender problems which involve fortification, interdiction and recourse decisions, respectively. These problems are interesting from a theoretical viewpoint and have many real-world applications, in, e.g., military operations or the design of robust networks. We introduce \emph{fortification cuts}, which allow for a single-level exact reformulation of this complex trilevel optimization problem,
\MM{and a solution approach which uses the} cuts within a branch-and-cut algorithm and works in the space of the fortification variables only.

While mainly focusing on the version where interdicting an asset depreciates its usefulness for the defender, we also show that our methodology is directly applicable to FGs with complete destruction of interdicted assets.
After introducing the basic fortification cuts, 
we present two \KT{strengthening} procedures that can be used to produce stronger inequalities, and give a detailed description of the separation procedure. 
We provide numerical results for two relevant test-cases, namely knapsack fortification games and shortest path fortification games. 
Due to the different structures and computational complexity of their recourse problems, we observe different effects of our algorithmic components in our computational study, and try to give insight about possible reasons.

Finally, we point out that our solution scheme is generic, and it leaves space for problem specific improvements. 
For example, for a given FG, one could solve the associated separation problem by means of
a state-of-the-art exact method or using some specialized heuristic instead of the greedy method proposed in this paper.
Moreover, problem-specific valid inequalities or \KT{strengthening} procedures of the fortification cuts could also be possible.

\section*{Acknowledgements}
This research was funded in whole, or in part,
by the Austrian Science Fund (FWF)[P 35160-N]. For the purpose of open access, the author has applied a CC BY public copyright licence to any Author Accepted Manuscript version arising from this submission. It is also supported supported by the Johannes Kepler University Linz, Linz Institute of Technology (LIT) (Project LIT-2019-7-YOU-211) and the JKU Business School. LIT is funded by the state of Upper Austria. The work of the third author was supported by the Air Force Office of Scientific Research under Grant no. FA8655-20-1-7012.

\bibliography{Fortification}

\end{document}